\documentclass{amsart}

\usepackage{amsmath,amssymb,amsthm}
\usepackage{graphicx,tikz}
\usetikzlibrary{calc}
\newtheorem{theorem}{Theorem}

\theoremstyle{definition}

\theoremstyle{remark}

\DeclareMathOperator{\erfc}{erfc}

\begin{document}

\title[]{Effective Bounds for the Decay of Schr\"odinger Eigenfunctions and Agmon bubbles}
\thanks{Partially supported by the NSF (DMS-2123224) and the Alfred P. Sloan Foundation.}

\author[]{Stefan Steinerberger}
\address{Department of Mathematics, University of Washington, Seattle, WA 98195, USA}
\email{steinerb@uw.edu}

\begin{abstract} We study solutions of $-\Delta u + V u = \lambda u$ on $\mathbb{R}^n$. Such solutions localize in the `allowed' region $\left\{x \in \mathbb{R}^n: V(x) \leq \lambda\right\}$ and decay exponentially
in the `forbidden' region $\left\{x \in \mathbb{R}^n: V(x) > \lambda\right\}$. One way of making this precise is Agmon's inequality 
implying decay estimates in terms of the Agmon metric. We prove a complementary decay estimate in terms of harmonic measure which can improve on Agmon's estimate, connect the Agmon metric to decay of harmonic measure and prove a sharp
pointwise Agmon estimate. 
\end{abstract}

\maketitle

\section{Introduction}
\subsection{Introduction.} We consider eigenfunctions of Schr\"odinger operators
$$ - \frac12 \Delta u(x) + V(x) u(x) = \lambda u(x), \qquad x \in \mathbb{R}^n,$$
where the potential $V(x) \geq 0$ is assumed to grow $V(x) \rightarrow \infty$ as $\|x\| \rightarrow \infty$.  $V$ could have extremely rapid growth outside some region $\Omega \subset \mathbb{R}^n$ which emulates
the Dirichlet problem in $\Omega$. 
 Multiplying
with $u$ and integration by parts leads to
$$ \frac12  \int_{\mathbb{R}^n} | \nabla u|^2 dx + \int_{\mathbb{R}^n} V(x)  u(x)^2 dx = \int_{\mathbb{R}^n} \lambda  u(x)^2 dx$$
This identity suggests that most of the $L^2-$mass of $u$ has to be contained in the `allowed' region $E = \left\{x \in \mathbb{R}^n: V(x) \leq \lambda \right\}$
and only very little can be contained in the `forbidden' region $\left\{x \in \mathbb{R}^n: V(x) > \lambda \right\}$. Agmon's approach \cite{agmon1} is now classical: define the \textit{Agmon metric} between two points $x,y \in \mathbb{R}^n$ as the minimum
energy taken over all paths from $x$ to $y$
$$ \rho_{\lambda}(x,y) = \inf_{\gamma} \int_0^1 \max\left( \sqrt{2} \sqrt{V(\gamma(t))- \lambda}, 0\right) |\dot \gamma(t)| dt,$$
where $\gamma:[0,1] \rightarrow \mathbb{R}^n$ ranges over all paths from $\gamma(0) = x$ to $\gamma(1) = y$. It is
understood that we can expect, in a suitable sense, for all $\varepsilon > 0$, that
$$ |u(x)| \leq c_{\varepsilon} \sup_{y \in \mathbb{R}^n \atop V(y) \leq \lambda} e^{-(1-\varepsilon) \rho_{\lambda}(x,y)}.$$
This decay estimate is phrased in terms of the minimal Agmon distance between the point $x$ and the allowed region $E= \left\{y \in \mathbb{R}^n: V(y) \leq \lambda\right\}$. We abbreviate
$$ \rho_{\lambda}(E, x) = \inf_{y \in E} \rho_{\lambda}(x,y)$$
allowing us to rephrase the Agmon estimate as
$$ |u(x)| \leq c_{\varepsilon} \exp\left(-(1-\varepsilon) \rho_{\lambda}(E,x)\right).$$
Note that many papers in the literature are concerned with the asymptotic decay as $\|x\| \rightarrow \infty$ and write the Agmon estimate in terms of $\rho_{\lambda}(0,x)$ which is equivalent up to constants. Since we are interested, among other things, in sharp estimates, we will always work with $ \rho_{\lambda}(E,x)$.
It is understood that the Agmon estimate may be very effective and can lead to optimal results. Consider the following simple example: $u:[0, \infty] \rightarrow \mathbb{R}$ given by
$u(x) = \exp(-x)$ satisfies the equation
$$ -\frac12 \Delta u(x) + V(x) u(x) = \lambda u(x) \qquad \mbox{with}~V(x) = \frac12~\mbox{and}~ \lambda=0.$$
We see that the Agmon metric is given by
$$   \rho_{\lambda}(x,0) = \int_0^{x} \sqrt{2} \sqrt{ \frac12 - 0} ~ dt = x$$
and we recover the exact equation
$ u(x) =  e^{-\rho_{\lambda}(0,x)}.$
However, it is also understood that the estimate may be ineffective and it is not difficult to construct such examples (see \S 4.3 for an extreme setting). We establish a new decay estimate (Theorem 1) which can complement and improve on Agmon's estimate in settings where it may be ineffective. We also prove that whenever Agmon's estimate is approximately sharp in a point $x$, then the harmonic measure is not too small: there has to be an entire tube around the optimal Agmon path (Theorem 2). 
Our approach also yields a sharp pointwise Agmon estimates (Theorem 3).

\subsection{The Agmon estimate.} 
We quickly discuss how Agmon estimates are often motivated: introducing the semi-classical scaling
$$ -h^2 \Delta u + Vu = \lambda u,$$
there is (assuming some moderate growth on $V$) an identity, valid for any (sufficiently regular) $\phi: \mathbb{R}^n \rightarrow \mathbb{R}$
$$ \int_{\mathbb{R}^n} \left| h \nabla  \exp\left( \frac{\phi}{h} \right) \right|^2 dx + \int_{\mathbb{R}^n} \left(V - \lambda - | \nabla \phi|^2 \right) \exp\left(\frac{2\phi}{h} \right) |u|^2 dx = 0.$$
This implies, in particular,
$$  \int_{\mathbb{R}^n} \left(V - \lambda - | \nabla \phi|^2 \right) \exp\left(\frac{2\phi}{h} \right) |u|^2 dx \leq 0.$$
What kind of $\phi$ would extract the most information? Since one of the factor is inside the exponential
function, we would like to make $\phi$ as large as possible without introducing negative quantities anywhere (since those could lead to cancellation). This means we want to ensure that $\phi$ is as large as possible while simultaneously satisfying
$ | \nabla \phi|^2 \leq V - \lambda.$
This, naturally, leads to the Agmon metric. Once this is established, one can deduce pointwise estimates can be established as follows: suppose $u$ is much larger than predicted by the Agmon estimate in a point $x_0$. Basic elliptic estimates imply that $u$ is somewhat comparable in a radius around $x_0$ which then leads to a substantial contribution to the integral. However, note that the underlying heart of the argument is clearly an integral identity.

\section{Main Results}
\subsection{Agmon bubbles.} We start by introducing a new decay estimate. Suppose  
$$ - \frac12\Delta u + V u = \lambda u \qquad \mbox{in}~\mathbb{R}^n.$$ Suppose furthermore that, for all points in the `forbidden' region $ \left\{x: V(x) > \lambda\right\}$, the solution $u$ satisfies, for some $\varepsilon > 0$, an Agmon estimate of the type
$$ |u(x)| \leq c_{\varepsilon} \exp\left((1-\varepsilon)\cdot \rho_{\lambda}(E,x) \right).$$  For any point $x$ in the `forbidden' region and any  $\alpha > \rho_{\lambda}(x, E)$, we define
$$\Omega_{\alpha} = \left\{y \in \mathbb{R}^n \setminus E: \rho_{\lambda}(E,y) \leq \alpha \right\} \subset \mathbb{R}^n \setminus E.$$
We interpret $\Omega_{\alpha} \subset \mathbb{R}^n$ as a (possible unbounded) subset of $\mathbb{R}^n$ and will refer to these
sets as `Agmon bubbles'.
Since  $\alpha > \rho_{\lambda}(x, E)$, we have $x \in \Omega_{\alpha}$.

\begin{theorem}[Agmon bubbles] We have, for any $\alpha > \rho_{\lambda}(x, E)$,
$$ \left|u(x)\right| \leq c_{\varepsilon} e^{-(1-\varepsilon) \alpha} + \omega_x^{(\Omega_{\alpha})}(\partial E) \cdot \|u\|_{L^{\infty}(\mathbb{R}^n)},$$
where $\omega^{(\Omega_{\alpha})}_x(\partial E)$ is the harmonic measure of $\partial E \subset \Omega_{\alpha}$ with respect to the pole $x$.
\end{theorem}

 The parameter $\alpha > \rho_{\lambda}(x, E)$
can be freely chosen and, in practical applications, it makes sense to optimize over it (see the subsequent examples).
We start by illustrating the result. Let us consider an example as in Fig. \ref{fig:bubb}: we have an allowed region $E$ which is surrounded by a porous ring on
which the potential $V$ assumes extremely large values. Since the ring is porous, there are Agmon paths which completely avoid the region where $V$ is
large: as long as there is any hole in the circle surrounding the disk, the Agmon estimate will not register it.  Theorem 1 provides a different type of estimate: since it measures porosity of the `ring' indirectly through harmonic measure, we obtain a suitable decay estimate.

\begin{center}
\begin{figure}[h!]
\begin{tikzpicture}[scale=1]
\draw[dashed] (6,0) circle (1.2cm);
\node at (6,0) {$E$};
\node at (6-2, -0.8) {$x$};
\draw[dashed] (6-2,-1) -- (6-0.7, -0.35);
\filldraw (6-2,-1) circle (0.04cm);
%\filldraw (4.8, 0.2) circle (0.2cm);
%\filldraw (4.8, -0.3) circle (0.2cm);
%\filldraw (5, -0.8) circle (0.2cm);
%\filldraw (5.5, -1) circle (0.2cm);
 \draw [ultra thick,domain=0:180] plot ({6+1.2*cos(\x)}, {1.2*sin(\x)});
  \draw [ultra thick,domain=190:200] plot ({6+1.2*cos(\x)}, {1.2*sin(\x)});
    \draw [ultra thick,domain=210:220] plot ({6+1.2*cos(\x)}, {1.2*sin(\x)});
        \draw [ultra thick,domain=230:240] plot ({6+1.2*cos(\x)}, {1.2*sin(\x)});
  \draw [ultra thick,domain=250:360] plot ({6+1.2*cos(\x)}, {1.2*sin(\x)});
   \draw [ultra thick,domain=0:180] plot ({6+1.23*cos(\x)}, {1.23*sin(\x)});
  \draw [ultra thick,domain=190:200] plot ({6+1.23*cos(\x)}, {1.23*sin(\x)});
    \draw [ultra thick,domain=210:220] plot ({6+1.23*cos(\x)}, {1.23*sin(\x)});
        \draw [ultra thick,domain=230:240] plot ({6+1.23*cos(\x)}, {1.23*sin(\x)});
  \draw [ultra thick,domain=250:360] plot ({6+1.23*cos(\x)}, {1.23*sin(\x)});
  \node at (3, 0.5) {`$V=\infty$'};
\draw [->] (3.7, 0.3) -- (4.5, 0.25);
\node at (12,0) {$E$};
\draw[dashed] (12,0) circle (1.2cm);
\node at (12-2, -0.8) {$x$};
\filldraw (12-2,-1) circle (0.04cm);
 \draw [ultra thick,domain=0:180] plot ({12+1.2*cos(\x)}, {1.2*sin(\x)});
  \draw [ultra thick,domain=190:200] plot ({12+1.2*cos(\x)}, {1.2*sin(\x)});
    \draw [ultra thick,domain=210:220] plot ({12+1.2*cos(\x)}, {1.2*sin(\x)});
        \draw [ultra thick,domain=230:240] plot ({12+1.2*cos(\x)}, {1.2*sin(\x)});
  \draw [ultra thick,domain=250:360] plot ({12+1.2*cos(\x)}, {1.2*sin(\x)});
   \draw [ultra thick,domain=0:180] plot ({12+1.23*cos(\x)}, {1.23*sin(\x)});
  \draw [ultra thick,domain=190:200] plot ({12+1.23*cos(\x)}, {1.23*sin(\x)});
    \draw [ultra thick,domain=210:220] plot ({12+1.23*cos(\x)}, {1.23*sin(\x)});
        \draw [ultra thick,domain=230:240] plot ({12+1.23*cos(\x)}, {1.23*sin(\x)});
  \draw [ultra thick,domain=250:360] plot ({12+1.23*cos(\x)}, {1.23*sin(\x)});
\draw [thick] (10, -0.4) to[out=320, in=180] (10.9, 0.5);
\draw [thick] (10, -0.4) to[out=130, in=90] (9.5, -1) to[out=270, in=180] (10, -1.3) to[out=0, in =255] (12.6, -1.1);
\node at (9.5, -1.5) {$\Omega_{\alpha}$};
\node at (9.5, 1.4) {$\rho_{\lambda}(E,\cdot) = \alpha$};
\draw [->] (9, 1) -- (9.6,-0.2);
\end{tikzpicture}
\caption{A sketch of an Agmon bubble attached to a point $x$.}
\label{fig:bubb}
\end{figure}
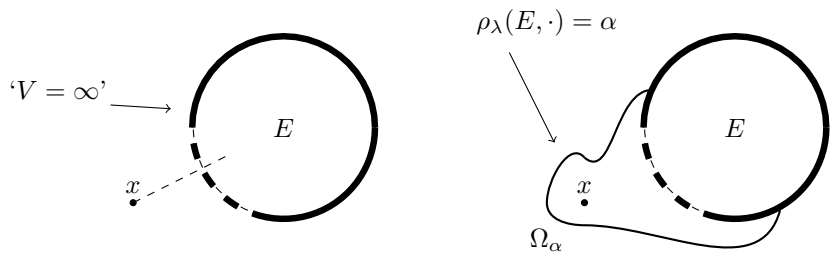
\end{center}

\textbf{An example.} 
We will now illustrate this in a simple case where all the computations can be done exactly. Consider the example 
$$ - \frac12 \Delta u + Vu = \lambda u\qquad \mbox{in}~\mathbb{R}^2,$$
where $\lambda = 0$, the allowed region $\left\{x: V(x) \leq \lambda \right\}$ is a ball around the origin of size $\sim 1$ and the potential
outside of $E$ is given by $V_{\varepsilon}(x,y) = \varepsilon + y^2$ (see Fig. \ref{fig:strips}). We expect $u$ to be localized in $E$. The Agmon distance is
approximately (depending a little on the precise shape and location of $E$), up to constants,
$$ \rho_{\lambda}(E, (x,y)) \sim \varepsilon |x| + \frac{y^2}{2}.$$
We see that Agmon's estimate implies relatively little decay in the $x-$direction, we only obtain $\sim \exp(-\varepsilon |x|)$. In particular, this becomes weaker and weaker as $\varepsilon \rightarrow 0$. One
could now wonder how accurate this can be: the rapid growth of the potential away from the $x-$axis should somehow enforce some type of uniform decay along the $x-$axis
even as $\varepsilon \rightarrow 0$.
\begin{center}
\begin{figure}[h!]
\begin{tikzpicture}[scale=1.3]
\draw [thick, ->] (0,0.5) -- (0,1);
\draw [thick, ->] (0,-0.5) -- (0,-1);
\draw [thick, ->] (0.5,0) -- (3.5,0);
\filldraw (2,0) circle (0.04cm);
\node at (2, -0.2) {$x$};
\draw[dashed] (0,0) circle (0.5cm);
\node at (0,0) {$E$};
\draw [<->] (-0.7, -0.5) -- (-0.7, 0.5);
\node at (-1, 0) {$\sim 1$};
\draw [dashed] (0.4, 0.7) -- (3.5, 0.7);
\draw [dashed] (0.4, -0.7) -- (3.5, -0.7);
\node at (2.8, 0.2) {$\Omega_{\alpha} \subset S_{\alpha}$};
\draw [<->] (5.2-1.5, -0.7) -- (5.2-1.5, 0.7);
\node at (5.8-1.5, 0) {$\sim\sqrt{8\alpha}$};
\end{tikzpicture}
\caption{The Agmon bubble is contained in an infinite strip.}
\label{fig:strips}
\end{figure}
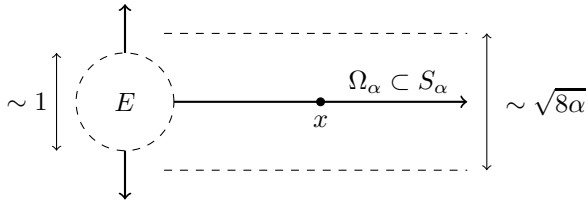
\end{center}

We will now apply Theorem 1, fix $x \gg 1$ and consider $u(x,0)$. Let $\alpha > 0$ be arbitrary
and consider the Agmon bubble
$$ \Omega_{\alpha} = \left\{(x,y) \in \mathbb{R}^2:  \rho_{\lambda}(E,x) \leq \alpha \right\}$$
which itself is contained in an infinite strip $S_{\alpha}$ of width $\sim \sqrt{8 \alpha}$. By simple domain monotonicity, we can bound
$$ \omega_x^{(\Omega_{\alpha})}(\partial E) \leq  \omega_x^{(S_{\alpha})}(\partial E).$$
 A standard estimate for harmonic measure implies
$$ \omega_x^{(S_{\alpha})}(\partial E) \leq  \exp \left( - c \frac{x}{  \sqrt{\alpha}}\right).$$
A version with explicit constant could, for example, be derived from a result of Betsakos \cite{bet}.
Using Theorem 1 and optimizing in $\alpha$ (setting $\alpha \sim x^{2/3}$) leads to
$$ |u(x,0)| \leq  e^{- c |x|^{2/3}} \cdot \|u\|_{L^{\infty}(\mathbb{R}^n)}$$
uniformly as $\varepsilon \rightarrow 0$.
This implies substantial decay in a regime where the classical Agmon estimate cannot deduce any information.\\

\textbf{Another Example.}
Consider the problem of understanding the behavior of the ground state shown in Fig. \ref{fig:row} which is comprised of four $1 \times 1$ squares on which $V$ assumes four different values. We assume that $V=\infty$ (or, somewhat equivalently, $V$ is extremely large) outside these four squares. The problem depends on the parameter
$m \gg 1$. As $m \rightarrow \infty$, we expect that the ground state of the problem localizes completely within the square where $V$ vanishes and that it
exhibits rapid decay into the three adjacent regions. 

\begin{center}
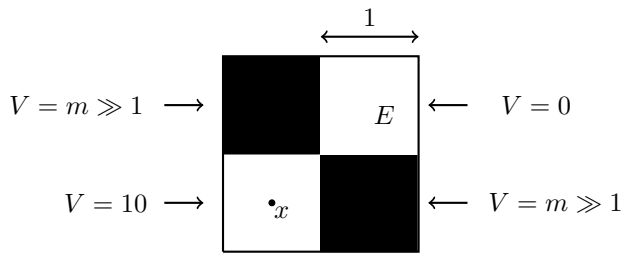
\begin{figure}[h!]
\begin{tikzpicture}[scale=1.3]
\draw [thick] (0,0) -- (2,0) -- (2,2) -- (0,2) -- (0,0);
\filldraw (1,0) rectangle ++(28pt,28pt);
\filldraw (0,1) rectangle ++(28pt,28pt);
\node at (3.4,0.5) {$V = m \gg 1$};
\draw [thick, ->] (2.5,0.5) -- (2.1, 0.5);
\node at (3.2,1.5) {$V=0$};
\draw [thick, ->] (2.5,1.5) -- (2.1, 1.5);
\filldraw (0.5, 0.5) circle (0.03cm);
\node at (0.6, 0.4) {$x$};
%\filldraw (1.5, 1.5) circle (0.03cm);
\node at (1.65, 1.4) {$E$};
\node at (-1.5,1.5) {$V = m \gg 1$};
\draw [thick,->] (-0.6, 1.5) -- (-0.2, 1.5);
\node at (-1.2,0.5) {$V = 10$};
\draw [thick,->] (-0.6, 0.5) -- (-0.2, 0.5);
\draw [thick,<->] (1, 2.2) -- (2, 2.2);
\node at (1.5, 2.4) {1};
\end{tikzpicture}
\caption{Using Theorem 1 to deduce decay for the ground state of this problem (imagine $V=\infty$ outside the square).}
\label{fig:row}
\end{figure}
\end{center}

\vspace{-20pt}

 Agmon's inequality immediately implies decay for the two squares in which $V = m \gg 1$ is very large. 
Agmon's inequality implies very mild decay for the region in which $V \equiv 10$ (this region is `mildly forbidden', we have $V > \lambda$, but barely so). However, the Agmon estimate does not imply stronger decay in that region as $m$ increases -- this is in contrast to classical intuition: as $m$ increases, the two squares in which $V \equiv m$ should start to act as an insulator and induce additional decay. 
We will now apply Theorem 1. First note that by using the ground state of a single square as a test function, we can immediately conclude that $\lambda_1(m) \leq \pi^2 < 10$. This allows us to conclude that the Agmon distance stays uniformly bounded in $m$
$$ \rho_{\lambda_1(m)}(E, x) \sim \sqrt{2}\sqrt{10 - \lambda_1(m)}d(x,E) \sim 1.$$
We first compute the Agmon bubble associated to the value $\rho_{\lambda_1(m)}(E, x) = \alpha$ where $\alpha$
is a parameter to be chosen later. 
The Agmon bubble will essentially contain the two squares where $V(x) \in \left\{0,10\right\}$
and it will contain a subregion of the squares where $V \equiv m$ is large. If $d\sigma$ denotes the arclength measure, then each step in
one of the two squares where $V \equiv m$ gets weighted by $\sqrt{m - \lambda_1(m)}~ d\sigma \sim \sqrt{m} ~d\sigma$. The bubble
$\Omega_{\alpha}$ will carve out a region of size $\sim \alpha/\sqrt{m}$ from the squares (see Fig. \ref{fig:megaexample}).

\begin{center}
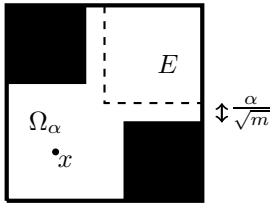
\begin{figure}[h!]
\begin{tikzpicture}[scale=1.3]
\draw [ultra thick] (0,0) -- (2,0) -- (2,2) -- (0,2) -- (0,0);
\filldraw (1.2,0) rectangle ++(23pt,23pt);
\filldraw (0,1.2) rectangle ++(23pt,23pt);
\filldraw (0.5, 0.5) circle (0.03cm);
\node at (0.6, 0.4) {$x$};
\node at (1.65, 1.4) {$E$};
\draw [thick, dashed] (1,2) -- (1,1) -- (2,1); 
\draw [thick, <->] (2.2, 1) -- (2.2, 0.8);
\node at (2.5, 0.9) {$ \frac{\alpha}{\sqrt{m}}$};
\node at (0.4, 0.8) {$\Omega_{\alpha}$};
\end{tikzpicture}
\caption{The Agmon bubble attached to the problem: $\Omega_{\alpha}$}
\label{fig:megaexample}
\end{figure}
\end{center}
\vspace{-20pt}

It remains to find an upper bound on the harmonic measure of $\partial E$ with respect to the point $x$: a simple consideration of the local
geometry implies the estimate 
$$\omega_x^{(\Omega_{\alpha})}(\partial E) \lesssim \alpha/\sqrt{m}$$ which, with Theorem 1, implies the decay estimate
\begin{align*}
 |u(x)| \lesssim c_{\varepsilon} e^{-(1-\varepsilon) \alpha} + \alpha/\sqrt{m}.
\end{align*}
Optimizing in $\alpha$ ($\alpha \sim c\log{m}$) then implies $$ |u(x)| \lesssim \frac{\log{m}}{\sqrt{m}}.$$

\subsection{Sharpness of Agmon's estimate.} Theorem 1 can, in particular, serve as a necessary condition illustrating what is required for the underlying equation to admit a sharp Agmon estimate: if $\exp\left( - \rho_{\lambda}(E,x)\right)$ is indeed roughly comparable to $|u(x)|$, then the harmonic measure of $\partial E$, when seen using the harmonic measure located in $x$ with respect to the Agmon bubble $\Omega_{\alpha}$ cannot be too small. In particular, it is then not possible that the Agmon distance is only realized by a path, there has to be an entire `tube' of paths with roughly comparable length around it.
Indeed, we can send $\alpha \rightarrow \infty$ in Theorem 1 and immediately deduce that for eigenfunctions with
$\| u\|_{L^{\infty}} \sim 1$, we have for all points in the forbidden region that
$$ \omega^{}_x( E) \gtrsim |u(x)|.$$

This is not terribly surprising: if the function $u$ was harmonic, these two quantities coincide. In the forbidden region,
we would expect the function $u$ to decay even faster than a harmonic function, so the estimate should be somewhat
lossy. Simultaneously, since $\Delta u = 2 (V(x) - \lambda) u$, it is certainly conceivable that $V(x)$ is just a tiny bit larger
than $\lambda$: then the function is nearly harmonic. We will now argue that as soon as we are strictly away from the
harmonic regime, there exists an improved estimate in terms of the Agmon metric. We will fix $\delta > 0$ and slightly
increase the allowed region to
$$ E_{\delta} = \left\{x \in \mathbb{R}^n: V(x) \leq \lambda + \delta\right\}.$$
\begin{theorem}[Agmon tube, informal] If the Agmon estimate is roughly sharp,
$$ |u(x)| \sim \exp\left( - \rho_{\lambda}(E_{\delta}, x)\right),$$
then the harmonic measure satisfies
$$ \omega^{}_x( E_{\delta}) \gtrsim \sqrt{\exp\left( - \rho_{\lambda}(E_{\delta}, x)\right)}.$$
\end{theorem}
This is a square root improvement over the trivial estimate.
We emphasize that this inequality shows another nontrivial connection between the geometry induced by the Agmon metric and 
 harmonic measure. It also emphasizes the point (see Fig. \ref{fig:paths}) that sharpness of the Agmon metric implies that the Agmon path
 may not be too sensitive: there has to be some neighborhood around it which implies the existence of paths with comparable length.
Theorem 2 is informally phrased and correct up to lower-order correction terms; the proof given in \S \ref{proof:2}. Agmon's estimate (which is actually re-proven as a byproduct of the argument). Using the same framework as Theorem 3,
it could be made precise up to constants.

\begin{center}
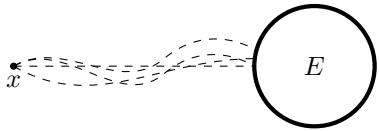
\begin{figure}[h!]
\begin{tikzpicture}[scale=1]
\draw[ultra thick] (6,0) circle (0.8cm);
\filldraw (2,0) circle (0.04cm);
\node at (6-4, -0.2) {$x$};
\draw[dashed] (6-4,0) -- (6-0.8, 0);
\node at (6,0) {$E$};
\draw [dashed] (2,0) to[out=30, in=220] (4, -0.1) to[out=40, in=150] (5.2,0);
\draw [dashed] (2,0) to[out=30, in=220] (4, 0.1) to[out=40, in=150] (5.25,0.2);
\draw [dashed] (2,0) to[out=330, in=180] (5.2, 0.1);
\end{tikzpicture}
\caption{Theorem 2: for Agmon's estimate to be sharp, there must be `many' short paths from $x$ to the allowed region (with `many' is measured by the harmonic measure $\omega_x(\partial E)$).}
\label{fig:paths}
\end{figure}
\end{center}

\vspace{-15pt}

\subsection{Pointwise Estimate.} As a byproduct of the approach, we obtain an explicit effective \textit{pointwise} Agmon estimate. Usually, sharp Agmon estimates show that for every $\varepsilon > 0$ there exists $c_{\varepsilon} > 0$ such that for all $x$ in the forbidden region
$$ |u(x)| \leq c_{\varepsilon} \inf_{y \in \mathbb{R}^n \atop V(y) \leq \lambda} e^{-(1-\varepsilon) \rho_{\lambda}(x,y)}.$$
It is explained in \S 1.2 where this formulation comes from and how it is natural when deducing pointwise estimates from an integral estimate. We were motivated by trying to understand under which assumptions precise pointwise Agmon estimates become possible. 
There is a rather natural condition that arises from the Agmon geometry: if $\rho_{\lambda}(E,x):\mathbb{R}^n \setminus E \rightarrow \mathbb{R}$ is superharmonic as a function of $x$, then we can obtain effective estimates without hidden constants. 
\begin{theorem}
Consider the Agmon metric as a function  $\rho_{\lambda}(E, x): \mathbb{R}^n \setminus E \rightarrow \mathbb{R}_{\geq 0}$. If $\rho_{\lambda}(E, x)$ is superharmonic in the forbidden region, 
$$ \Delta \rho_{\lambda}(E, x)  \geq 0 \qquad \mbox{on}~\mathbb{R}^n\setminus E,$$
then we have the pointwise inequality 
$$  \forall~x \in \mathbb{R}^n \setminus E: \qquad |u(x)| \leq   \|u\|_{L^{\infty}(\partial E)} \cdot e^{-\rho_{\lambda}(E, x)}.$$
\end{theorem}
 A simple computation shows that if the potential $V$ is radial, then the condition is automatically satisfied for monotonically increasing potentials (and even for some decreasing potentials as long as they do not decrease too quickly). 
The inequality is sharp: consider $n=1$ and consider the problem
$$ -\frac12 \Delta u + \frac12 u = 0,$$
i.e. $V=1/2$ and $\lambda = 0$. This problem has the solution $u(x) = e^{-x}$ on $[0,\infty]$. We may interpret, by locally modifying the potential, that the allowed region is given by $E = \left\{0 \right\}$. The Agmon distance is then given by $\rho_{\lambda}(E,x) = x$ (which certainly satisfies $\Delta \rho_{\lambda}(E,x) \geq 0$) and we see that the inequality is sharp. There are many possible variations on Theorem 3 under different conditions on $\Delta \rho_{\lambda}(E, x)$ that follow from the same argument, we refer to \S 5.4 for details.
%\begin{theorem}
%Consider the Agmon metric as a function $\rho_{\lambda}(E, x): \mathbb{R}^n \setminus E \rightarrow \mathbb{R}_{\geq 0}$. If the Agmon metric satisfies, for some constant $\nu < 0$,
%$$  \forall~x \in \mathbb{R}^n \setminus E: \qquad   \frac{\Delta  \rho_{\lambda}(E,x)}{V(x) - \lambda}  \geq   \frac{2\nu + 1}{\rho_{\lambda}(E,x)},$$
%then we have the pointwise inequality, for all $x \in \mathbb{R}^n \setminus E$,
%$$ |u(x)| \leq   \|u\|_{L^{\infty}(\partial E)} \cdot \frac{2^{\nu+1}}{\Gamma(|\nu|)}  \cdot  \frac{K_{\nu}(\rho_{\lambda}(E,x))}{\rho_{\lambda}(E,x)^{\nu}},$$
%where $K_{\nu}$ is the modified Bessel function of the second kind.
%\end{theorem}
%
%We note that asymptotic expansion for $K_{\nu}$ are readily available and, up to leading order, for $\nu$ fixed and $x \rightarrow \infty$, we have
%$$ K_{\nu}(x) = \sqrt{\frac{\pi}{2}} \frac{e^{-x}}{\sqrt{x}} + \mbox{lower order terms}.$$
%We see that the inequality thus recovers, up to polynomial corrections, the correct Agmon scaling. Note also that for $-1/2 < \nu < 0$, we have a quantified version of superharmonicity and the polynomial corrections lead to increased polynomial decay over the classical Agmon estimate. For $\nu = -1/2$, we recover Theorem 2.

\subsection{Examples.} We conclude with a quick geometric description of what it means for $\Delta \rho_{\lambda}(E, x)$ to have a certain size. Given the nontrivial definition of $\Delta \rho_{\lambda}(E, x)$, one would perhaps not expect a complete geometric description -- our goal will be to paint a suitable picture that conveys the main idea. 
We start with a one-dimensional example (which also applies to higher-dimensional constructions with Cartesian product structure). In the one-dimensional case, assuming for the forbidden region to contain $[0,\infty]$, we have 
$$ \rho_{\lambda}(0, x) = \int_0^x \sqrt{2} \sqrt{V(y) - \lambda} ~dy \qquad \mbox{and} \qquad  \Delta \rho_{\lambda}(0, x) =  \sqrt{2} \frac{d}{dx}\sqrt{V(x) - \lambda}$$
which increases if $V$ increases and decreases otherwise. $ \Delta \rho_{\lambda}(0, x) \geq 0$ is simply a statement about the second derivative of a potential. We will now fix $V(x) - \lambda = c$ to be constant and consider a two-dimensional example shown in Fig. \ref{fig:geome} (left): the forbidden region is inside a sphere of radius 1, the allowed region $E$ is outside the sphere. Then, for $\|x\| \leq 1$ inside the sphere, we have
$$ \rho_{\lambda}(E,x) = \sqrt{2c} \cdot (1-\|x\|) \qquad \mbox{and} \qquad \Delta \rho_{\lambda}(E,x) = - \frac{\sqrt{2c}}{\|x\|}.$$
In particular, we note that the quantity is negative: because of the curvature of the sphere, one is always slightly closer to $E$ than one would be if $E$ was bounded by, say, a hyperplane. The effect is strongest in the origin. One would expect $\Delta \rho_{\lambda}(E,x) < 0$ in regions where $E$ is accessible from various directions.

\begin{center}
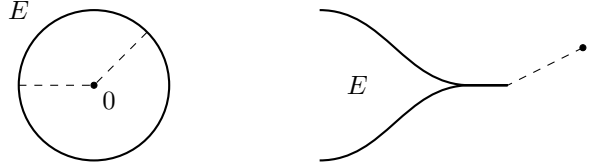
\begin{figure}[h!]
\begin{tikzpicture}[scale=1]
\draw [thick] (0,0) circle (1cm);
\filldraw (0,0) circle (0.04cm);
\node at (0.2, -0.2) {$0$};
\draw [dashed] (0.707, 0.707) -- (0,0) -- (-1,0);
\node at (-1,1) {$E$};
\draw [thick] (3,1) to[out=0, in=180] (5,0) --(5.5,0) -- (5,0) to[out=180, in=0] (3,-1);
\node at (3.5,0) {$E$};
\filldraw (6.5, 0.5) circle (0.04cm);
\draw [dashed] (6.5, 0.5) -- (5.5, 0);
\end{tikzpicture}
\caption{Left: at the center of a sphere of radius 1, the allowed region is outside the sphere. Right: a small part of the allowed region $E$ is relatively isolated but dominates the Agmon distance.}
\label{fig:geome}
\end{figure}
\end{center}

As for the example of a positive Laplacian, we consider Fig. \ref{fig:geome} (right): the allowed region $E$ has a small part that moves far into
the forbidden region. We note that, under fairly general condition, we can then expect $\Delta \rho_{\lambda}(E,x) \geq 0$: since $\partial E$ is only
accessible from very specific directions, it is `further away' than it would be if it were more spread out.
We note that this last example also serves as yet another illustration of how Theorem 1 can be profitably applied: harmonic measure often has a
better global understanding of the size of a set than the function assigning distance to the nearest point from that set.
Suppose the allowed region $E$ has relatively little volume in a certain region of space (see Fig. \ref{fig:geome2}): the set
$E = \left\{x: V(x) \leq \lambda\right\}$ is present in a certain region but is simultaneously so small in terms of volume that it cannot affect
the solution of the PDE very much. It will, however, strongly affect the computation of the Agmon distance since there are now, in this
region of space, many short paths to the set $E$. The Agmon distance can only predict relatively little decay. Consider applying Theorem 1: once $d \geq 3$, the harmonic measure with respect to such long elongated segments will be very small and it will, in contrast to the Agmon metric,
not contribute much. This means that harmonic measure will put more weight on the other parts of $\partial \Omega_{\alpha}$ which then induces additional decay.

 \begin{center}
\begin{figure}[h!]
\begin{tikzpicture}[scale=1]
\draw [ultra thick] (3,1) to[out=0, in=180] (4,0) --(5.5,0) -- (4,0) to[out=180, in=0] (3,-1);
\node at (3.3,0) {$E$};
\filldraw (6, 0.3) circle (0.04cm);
\draw [dashed] (6, 0.3) -- (5.5, 0);
\draw [ultra thick] (3+5,1) to[out=0, in=180] (4+5,0)  to[out=180, in=0] (3+5,-1);
\draw [ultra thick, dashed] (9,0) -- (10.5,0);
\node at (3.3+5,0) {$E$};
\filldraw (6+5, 0.3) circle (0.04cm);
\draw [dashed] (6+5, 0.3) -- (5.5+5, 0);
\draw [thick] (8.5, 1.25) to[out=0, in=90] (12, 0) to[out=270, in=0] (8.5, -1.25);
\node at (10, -0.7) {$\Omega_{\alpha}$};
\end{tikzpicture}
\caption{Applying Theorem 1: there is a long but nearly empty part of the allowed region $E$ dominating the notion of the Agmon distance in that region. The long elongated part becomes nearly invisible through the lense of the harmonic measure when $d \geq 3$.}
\label{fig:geome2}
\end{figure}
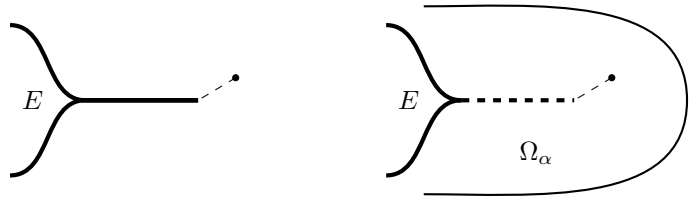
\end{center}

\subsection{Related Results.} We conclude with a short discussion of some of the related literature. The importance of the problem in mathematical physics
has lead to a plethora of results that would be difficult to summarize here (we refer to the book of Agmon \cite{agmon1}, the review by Deift \cite{deift}, the more recent survey of Hislop \cite{hislop} and references therein).
From the perspective of our paper, the story begins with the Agmon metric introduced in Agmon \cite{agmon1}, a precursor can be found in the work of Lithner \cite{lithner}.
The second main ingredient is the probabilistic approach to these questions, we refer to Carmona \cite{carmona0} and Carmona \& Simon \cite{carmona}.
 Carmona \cite{carmona0} uses the Feynman-Kac formula to reinterpret the eigenfunction problem and to both simplify and and improve on earlier results.
Carmona  \& Simon \cite{carmona} use the probabilistic approach to study a problem that is somewhat dual to ours: they use the framework to deduce that Agmon's estimate produces an essentially sharp \textit{lower} bound for the ground state; the main reason why the method is restricted to the ground state  is that one wants to avoid possible cancellation in the path integral and the ground state does not change sign. However, as is clear form the framework and as pointed out by Carmona \& Simon \cite{carmona}, the argument is certainly more versatile than that (since the main contribution comes from one path, the problem naturally localizes along paths). At first, this optimality result may seem to contradict our Theorem 1 which, in a manner of speaking, says that Agmon's estimate is \textit{only} optimal if there is an associated Agmon bubble that is sufficiently large (ground state or not). This discrepancy is resolved by a careful inspection of \cite[Eqs. (2.5) -- (2.9)]{carmona}: the sharpness is with respect to a slightly different Agmon metric associated to a slightly different potential $V_{+}$ (which arises from $V$ by taking maxima over nearby values). Some modification like this is presumably necessary: a well-known open problem of Landis \cite{landis} is as follows: consider
$ -\Delta u = V u$ and  $|V(x)| \leq 1.$
Suppose, moreover, that the solution assumes its global maximum in $x=0$. Is it possible that $u$ decays faster than $\exp(-\|x\|^{1+\varepsilon})$? This is known to be false if $V$ is complex-valued \cite{meshkov} but open if $V$ is real-valued. Note that $\exp(-\|x\|)$ is exactly the rate prediced by the Agmon estimate; thus, understanding precise conditions for the sharpness of the Agmon estimate is intricately linked to difficult problems in PDE.
To the best of our knowledge, our results are new and nothing similar is known in the literature. There are relatively few papers that use the probabilistic approach: a notable exception is a paper of Aizenman \& Simon \cite{aiz}.  Some of our results are philosophically related to earlier results of the author \cite{stein} which was motivated by the use of the Filoche-Mayboroda landscape function \cite{filoche} (a path integral interpretation of which is given in \cite{stein1}).  A generalization of Carmona \cite{carmona0} in the context of RCD spaces was recently given by Thalmaier \& Thompson \cite{thal}.
Denisov \& Kupin \cite{denisov} discuss the connection between diffusions, capacity and the behavior of the Schr\"odinger operator.

\section{Proofs: some preliminary aspects}
\subsection{Initial considerations.}
This section discuss aspects that will be common to all our arguments. We have the equation
$$ - \frac{1}{2} \Delta u + V u = \lambda u.$$
Therefore $u$ is a stationary solution of the heat equation
$$ \frac{\partial u}{\partial t} = \frac{1}{2} \Delta u  + \left( \lambda -  V\right)u.$$
The crucial ingredient of our argument is an analysis of the identity
$$ u(x) = \mathbb{E}_x \left(u(\omega_x(t)) e^{   \lambda t -   \int_0^t V(\omega_x(s)) ds} \right),$$
where $\omega_x(t)$ denotes Brownian motion started in $x$ after $t$ units of time. This is partially
what motivates our scaling of the Laplacian as $\Delta/2$: this scaling makes it the infinitesimal
generator of standard Brownian motion. 
This particular
formulation of the identity, i.e.
$$ u(x) = \mathbb{E}_x \left(u(\omega_x(t)) e^{   \lambda t -   \int_0^t V(\omega_x(s)) ds} \right),$$
will prove to be a little bit cumbersome and we will work with a rescaled version of this
identity which we now motivate. We first start by noting that we can assume w.l.o.g that $u(x) > 0$: the
equation is invariant under multiplication by $(-1)$. We always assume $u(x) > 0$ in the point $x$ in which
we wish to obtain a decay estimate and that $x$ is in the forbidden region, i.e. $V(x) > \lambda$. Note that 
$$ \frac12 \Delta u = (V(x) - \lambda) u(x)> 0 \qquad \mbox{in the forbidden region}$$
which means that there is always a direction of ascent along which $u$ is increasing: this direction of ascent has to end somewhere on $\partial E$, the boundary of the allowed region. In particular, for each $x$ in the forbidden region, we have $|u(x)| \leq \|u\|_{L^{\infty}(\partial E)}$.

\subsection{A heuristic rescaling.} Consider small values $0 \leq t \ll 1$. We expect $\omega_x(t)$ to be distance $\sim \sqrt{t}$ from $x$ (and distributed like a Gaussian centered at $x$).
We also expect, due to the continuity of $V$, that for very small values of $t$
$$  \lambda t -   \int_0^t V(\omega_x(s)) ds = t \cdot (\lambda - V(x_0)) + o(t).$$
This leads us to conclude that, for $0 < t \ll 1$,
\begin{align*}
 u(x) &= \mathbb{E}_x \left(u(\omega_x(t)) e^{   \lambda t -   \int_0^t V(\omega_x(s)) ds} \right) \\
 &\sim \mathbb{E}_x \left(u(\omega_x(t))  e^{t \cdot (\lambda - V(x_0)) + o(t)}\right) \sim e^{t \cdot (\lambda - V(x_0)) + o(t)} \cdot \mathbb{E}_x \left(u(\omega_x(t))  \right).
\end{align*}
This tells us that we should expect $\mathbb{E}_x u(\omega_x(t))$ to be a little bit bigger than $u(x)$. This is perhaps not surprising: in the forbidden region, we have
$$ \frac12\Delta u(x) = ( V(x) - \lambda) u(x) > 0$$
which means that for small balls $B_r(x)$ centered at $x$, the average value of $u$ on $B_r(x)$ is slightly larger than the value in $x$. In particular, for $t$ small
$$ \mathbb{E}~ u(\omega_x(t)) = e^{ (V(x_0) - \lambda) t + o(t)} u(x).$$
This identity suggests a natural rescaling in time.
Let us fix a point $x_0$, pick a very small $0 < \delta \ll 1$ and consider the process of running Brownian motion for time
$ t_{x_0} = \delta/(V(x_0) - \lambda)$.
Then, we expect the local change in the value to behave as in
$$ \mathbb{E}~ u(\omega_{x_0}(t_{x_0})) = e^{\delta + o(\delta)} u(x_0).$$
At this point, this may seem a bit arbitrary, one can always rescale in any way one wishes -- the clear advantage of this particular rescaling becomes more apparent when we try to understand how the Agmon metric 
changes with respect to running Brownian motion for this particular amount of time. 
Let us assume there exists a path $\gamma$ from the allowed region to $E$ such that
$$ \rho_{\lambda}(E, x_0) = \inf_{\gamma} \int_0^1 \max\left( \sqrt{2} \sqrt{V(\gamma(t))- \lambda}, 0\right) |\dot \gamma(t)| dt$$
and let us assume that this minimizing path is sufficiently smooth in a neighborhood of $x_0$ for us to define a tangent
direction $\dot \gamma(1)$.
Assuming the Agmon metric to be
$\rho_{\lambda}(E, x)$ to be locally differentiable around $x=x_0$, it becomes easy to analyze the behavior of the Agmon metric
under this rescaled Brownian motion. For small directions $\varepsilon \cdot y_0$, where $\varepsilon > 0$ and $\|y_0\|=1$, one obtains to leading order
$$ \rho_{\lambda}(E, x_0 + \varepsilon y_0) =  \rho_{\lambda}(E, x_0)  + \left\langle \frac{\dot\gamma(1)}{\| \dot \gamma(1)\|}, \varepsilon y_0 \right\rangle \sqrt{2} \sqrt{V(x_0) - \lambda} + \mbox{l.o.t.}$$

\begin{center}
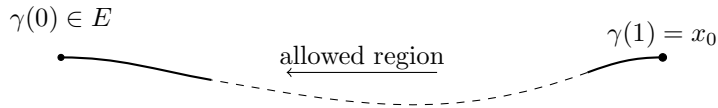
\begin{figure}[h!]
\begin{tikzpicture}[scale=1]
\filldraw (1,0) circle (0.05cm);
\node at (-3,0) {allowed region};
\draw[->] (-2, -0.2) -- (-4,-0.2);
\draw [thick] (0, -0.2) to[out=20,in=180] (1,0) to[out=0, in=200] (1, 0);
\draw [dashed] (0, -0.2) to[out=200,in=350] (-5,-0.3);
\draw [thick] (-5, -0.3) to[out=170, in=0] (-7, 0);
\filldraw (-7, 0) circle (0.04cm);
\node at (1, 0.3) {$\gamma(1) = x_0$};
\node at (-7, 0.5) {$\gamma(0) \in E$};
\end{tikzpicture}
\caption{A local analysis around $x_0$.}
\end{figure}
\end{center}

The inner product of an $n-$dimensional Brownian motion against any vector of length 1 behaves like a one-dimensional Brownian motion. Therefore, we can
expect to leading order for $t$ small that the Agmon distance of Brownian motion is
$$ \rho_{\lambda}(E, \omega_{x_0}(t)) =  \rho_{\lambda}(E, x_0)  + \omega^1_0(t) \sqrt{2} \sqrt{V(x_0) - \lambda} + \mbox{l.o.t,}$$
where $\omega_0^1(t)$ denotes a standard one-dimensional Brownian motion. 
However, because of our special choice of time $ t_{x_0} = \delta/(V(x_0) - \lambda)$,
we end up with
$$ \rho_{\lambda}(E, \omega_{x_0}(t_0)) \sim  \rho_{\lambda}(E, x_0)  + \omega_0^1(t)\sqrt{2} + \mbox{l.o.t.}$$ 
Note that $\omega_0^1(t)\sqrt{2}  = \omega_0^1(2t)$.
We can summarize this insight as follows: by suitable rescaling the running time of Brownian motion, we arrive at a diffusion process $\omega^*$ which satisfies
$$ \mathbb{E}~ u(\omega_{x_0}^*(t_{x_0})) = e^{\delta + o(\delta)} u(x_0)$$
and whose behavior under the Agmon metric behaves, up to first order, like a classical isotropic diffusion in one dimension in the sense of
$$ \rho_{\lambda}(E, \omega^*_{x_0}(t_0)) \sim  \rho_{\lambda}(E, x_0)  +\omega_0^1(2 t_0) + \mbox{l.o.t.}$$ 
\begin{center}
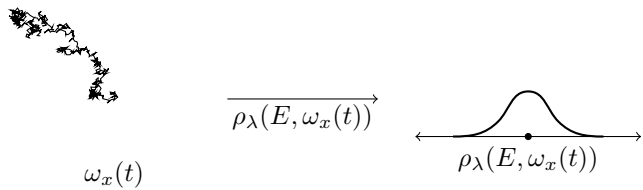
\begin{figure}[h!]
\begin{tikzpicture}[scale=1]
\draw (0,0) \foreach \x in {1,...,500}{--++(rand*0.05,rand*0.05)};
\node at (0,-1) {$\omega_x(t)$};
\draw [->] (1.5, 0) -- (3.5,0);
\node at (2.5,-0.25) {$\rho_{\lambda}(E, \omega_x(t))$};
\draw[<->] (4, -0.5) -- (7, -0.5);
\filldraw (5.5, -0.5) circle (0.04cm);
\node at (5.5, -0.8) {$\rho_{\lambda}(E, \omega_x(t))$};
\draw [thick] (4.5, -0.5) to[out=0, in=225] (5.1, -0.3) to[out=45, in=180] (5.5, 0.1) to[out=0, in=135] (5.9, -0.3) to[out=315, in=180] (6.5, -0.5);
\end{tikzpicture}
\caption{A common ingredient in all these arguments: rescaled Brownian motion behaves with respect to its Agmon distance, to leading order, like classical diffusion on the real line.}
\end{figure}
\end{center}

\subsection{Formalizing the heuristic.}
This line of reasoning therefore suggests that we rescale the main equation as
$$  - \frac{1}{2} \frac{1}{V - \lambda}\Delta u =  u \qquad \mbox{in the forbidden region.}$$
We emphasize that this PDE is only meaningful in the forbidden region since the coefficient blows up as we approach the $\left\{x: V(x) = \lambda\right\}$. It 
is thus meaningful to attach boundary conditions to the boundary of the forbidden region $ \left\{x: V(x) = \lambda\right\}$. Perhaps unsurprisingly, we will always assume that these boundary conditions are
given by $u$ (the values of the true eigenfunction): this ensures that the solutions of both equations (the eigenfunction and the solution of the rescaled equation) coincide in the forbidden region.
Attached to this equation is the parabolic equation
$$\frac{\partial u}{\partial t} -\frac{1}{2} \frac{1}{V - \lambda}\Delta u = u \qquad \mbox{inside the forbidden region}$$
with boundary conditions given by $u$ on the boundary of the forbidden region.
It is easily seen that if the initial condition is given by $u(0,x) = u(x)$, then the solution remains invariant under time and $u(t,x) = u(x)$ for all $t \geq 0$: $u$ is a stationary solution. At the same time, the symbol $ - (1/2)(V-\lambda)^{-1} \Delta$ can be interpreted as the infinitesimal generator of a random process
$\omega^*_{x_0}(t)$ which behaves exactly like Brownian motion rescaled by a factor of $(V(x) - \lambda)^{-1/2}$ in a point $x$. Then
$$ u(x) = \mathbb{E} \left[u(\omega^*_{x_0}(t \wedge \tau)) e^{- (t \wedge \tau)} \right]$$
where $t \wedge \tau = \min(t, \tau)$ and $\tau$ is the stopping time for the first impact on the boundary of the allowed region, $\partial E = \left\{x: V(x) = \lambda\right\}$ on which we have prescribed Dirichlet boundary conditions $u(t,x) = u(x)$.

\subsection{A stochastic inequality.} These arguments lead to a natural pointwise inequality (Theorem 4): we will show (as a byproduct of the proof of Theorem 2) that this simple inequality immediately implies the classical Agmon estimate but also allows for a finer analysis (leading to Theorem 3). We define a stochastic process $\omega_x^*(t)$ for $x$ in the forbidden region as motivated in \S 3.3: for any $x$ such that $V(x) > \lambda$, we define $\omega_x^*(0) = x$ and then, for positive times, as a Brownian motion whose
infinitesimal generator is given by $\Delta/(2V(x) - 2\lambda).$
This process behaves like a rescaled Brownian motion: it moves more quickly when $V(x) - \lambda$ is small and it moves more slowly when $V(x) - \lambda$ is large. For any $x$ in the forbidden region, we can now define a stopping time $\tau_x$ as the first time for which $\omega_x^*(\tau_x) \in \partial E$. Note that $\tau_x \geq 0$ and that $\tau_x = \infty$ is possible.

\begin{theorem} Using this notation, for any $x$ in the forbidden region, 
$$ |u(x)| \leq \|u\|_{L^{\infty}(\partial E)} \cdot \mathbb{E}\left[ e^{-\tau_x} \right].$$
\end{theorem}

\begin{proof}
 Using
 $$ u(x) = \mathbb{E} \left[u(\omega^*_{x_0}(t \wedge \tau)) e^{- (t \wedge \tau)} \right],$$
we may trivially estimate, by letting $t \rightarrow \infty$,
 \begin{align*}
  \left|u(x)\right| &= \left| \mathbb{E} ~u(\omega^*_{x_0}(t \wedge \tau)) e^{- (t \wedge \tau)}  \right| \\
  &\leq  \mathbb{E}   \left| u(\omega^*_{x_0}( \tau)) e^{-  \tau}  \right| \leq \|u\|_{L^{\infty}(\partial E)} \cdot \mathbb{E}(e^{-\tau}).
  \end{align*}
  \end{proof}
 We emphasize that with regards to exponential decay of the eigenfunction, then this inequality will usually not lose too much (this is also evidenced by it implying Theorem 2 and Theorem 3). A priori, the stopping time of an inhomogeneous diffusion process sounds too difficult to control, however, most of the difficulty gets neatly encapsulated in the definition of the Agmon metric (see \S 3.2).

\section{Proof of Theorem 1} 
\subsection{Proof of Theorem 1}
\begin{proof} Fix a point $x$ in the forbidden region $x \in \mathbb{R}^n \setminus E$, fix an arbitrary $\alpha >  \rho_{\lambda}(E,x) $ and define the Agmon bubble as
$$\Omega_{\alpha}  = \left\{ y \in \mathbb{R}^n \setminus E: \rho_{\lambda}(E,y) \leq \alpha \right\}.$$
The boundary of $\Omega_{\alpha}$ is given by a subset of $\partial E$ and a set of points for which $\rho_{\lambda}(E,y) = \alpha$.
We will study the behavior of the equation
$$\frac{\partial u}{\partial t} -\frac{1}{2} \frac{1}{V - \lambda}\Delta u = u \qquad \mbox{inside} ~\Omega_{\alpha},$$
where the initial values are given by $u$ and the boundary values, for all times $t>0$, are given by $u$. The solution of the equation is $u(t,x) = u(x)$. Now
$$ u(x) = \mathbb{E} \left[u(\omega^*_{x}(t \wedge \tau)) e^{- (t \wedge \tau)} \right],$$
where $\tau$ is the first hitting time of the particle on the boundary $\partial \Omega_{\alpha}$ (see \S 3.3). Letting $t \rightarrow \infty$ we estimate
$$ \mathbb{E} \left[u(\omega^*_{x}(t \wedge \tau)) e^{- (t \wedge \tau)} \right] \leq  \mathbb{E} \left|u(\omega^*_{x}( \tau)) \right|.$$

%\begin{center}
%\begin{figure}[h!]
%\begin{tikzpicture}
%\draw [dashed, thick] (0,0) circle (1cm);
%\node at (0,0) {$E$};
%\filldraw (3,0) circle (0.04cm);
%\node at (3.2, -0.2) {$x$};
%\draw (1.3, 0.1) to[out=315, in =90] (3.5, 0) to[out=270, in =45] (1.3, -0.5);
%\draw (-1.2, 0) to[out=90, in =180] (0, 1.2) to[out=0, in =135] (1.3, 0.1);
%\draw (-1.2, 0.1) to[out=270, in =180] (0, -1.2) to[out=0, in =225] (1.3, -0.5);
%\node at (3.8, 0.3) {$\Omega_{\alpha}$};
%\end{tikzpicture}
%\caption{Where do Brownian particles started in $x$ end up?}
%\end{figure}
%\end{center}
The remaining question is now simple: where do the points end up when hitting $\partial \Omega_{\alpha}$? There are two cases: they end up on
the boundary of the allowed region $\partial E \cap \partial \Omega_{\alpha}$ or on the boundary of $\partial \Omega_{\alpha} \setminus \partial E$. The two cases are
$$ \rho_{\lambda}(E, \omega^*_{x}( \tau)) = 0 \qquad \mbox{and} \qquad \rho_{\lambda}(E, \omega^*_{x}( \tau)) = \alpha.$$
We have
\begin{align*}
 \mathbb{E} \left|u(\omega^*_{x}( \tau)) \right| &=  \mathbb{E}  \left( \left|u(\omega^*_{x}( \tau)) \right| \big|\rho_{\lambda}(E, \omega^*_{x}( \tau)) = 0 \right)\cdot \mathbb{P}\left(\rho_{\lambda}(E, \omega^*_{x}( \tau)) = 0 \right) \\
 &+ \mathbb{E}  \left( \left|u(\omega^*_{x}( \tau)) \right| \big|\rho_{\lambda}(E, \omega^*_{x}(\tau)) = \alpha \right)\cdot \mathbb{P}\left(\rho_{\lambda}(E, \omega^*_{x}( \tau)) = \alpha \right).
\end{align*}
At this point, we recall that $\omega^*_x(t)$ is slightly different from the classical Brownian motion $\omega_x(t)$ since its local speed of propagation depends on the size of $V(x) - \lambda$. However, while this may lead to locally larger or smaller step-sizes, it has no impact on the direction of the
random walk, it only corresponds to a change of time. Alternatively, one could interpret it as a local change of the metric and appeal to the fact that harmonic measure is invariant under conformal changes of a metric: in any case, the harmonic measure with respect to $\omega^*_x$ and $\omega_x$ coincides.
Recalling the definition of harmonic measure
\begin{align*}
\mathbb{P}\left(\rho_{\lambda}(E, \omega^*_{x}( \tau)) = 0 \right) = \omega_x^{(\Omega_{\alpha})}(\partial E) \end{align*}
this inequality simplifies to
$$ |u(x)| \leq \sup_{x \in \partial \Omega \setminus \partial E} |u(x)| + \omega_x^{(\Omega_{\alpha})}(\partial E) \cdot \|u\|_{L^{\infty}}.$$
Applying the Agmon estimate on $\partial \Omega \setminus \partial E$, we arrive that
$$ |u(x)| \leq c_{\varepsilon} e^{- (1-\varepsilon) \alpha} + \omega_x^{(\Omega_{\alpha})}(\partial E) \cdot \|u\|_{L^{\infty}(\mathbb{R}^n)}.$$
\end{proof}

\subsection{Champagne regions.}
This subsection illustrates a class of examples where Agmon's estimate fails to capture the geometry of the eigenfunction in a very strong sense: so-called champagne regions. 
 Denoting a ball of radius $r$ centered at $x$ by $B(x,r)$, champagne domains are domains of the form
 $$ \Omega = \mathbb{R}^n \setminus \bigcup_{i=1}^{\infty} B(x_i, r_i).$$
 It is, for simplicity, often assumed that $0 \in \Omega$. One may think of the $B(x_i, r_i)$ as little bubbles (hence the name). There is now a rather
 natural question: is there a positive probability that Brownian motion started in 0 never hits any of the balls? In terms of PDEs, the question is whether the solution of
 \begin{align*}
 \Delta u &= 0 \qquad \mbox{in}~\Omega\\
 u&= 0 \qquad \mbox{at}~\partial \Omega \\
 u &= 1 \qquad \mbox{at infinity}
 \end{align*}
 satisfies $u(0) > 0$. This will, naturally, depend on the configuration of the balls: if they form an impenetrable barrier between 0 and $\infty$, then it is naturally not possible. The main insight is that other examples exist: as long as there are sufficiently many small balls, the likelihood may simply be 0. 
Akeroyd \cite{ake} showed that for any $\varepsilon > 0$ there exists a champagne subregion of the unit disk $B(0,1) \subset \mathbb{R}^2$ such that
$$ \sum_{i=1}^{\infty} r_k \leq \varepsilon$$
and nonetheless one of the bubbles will be hit with likelihood 1 by Brownian motion started in the origin. We also refer to subsequent work of
Carroll \& Ortega-Cerda \cite{carroll}, Gardiner \& Ghergu \cite{gardiner} and Ortega-Cerda \& Seip \cite{ortega}. The relevance for our problem is as follows: consider a champagne domain and make the potential $V$ very large inside the bubbles. Then the Agmon metric will not be able to conclude decay: after all, there are still short paths from most points outside the bubbles to most others. In contrast, the Agmon bubble $\Omega_{\alpha}$ will inherit the champagne structure and the harmonic measure of $\partial E$ will be as small as one wishes (therefore, with Theorem 1, capturing decay that is invisible to the standard Agmon metric). 

\section{Proof of Theorem 3}
\subsection{The Idea.}
The idea behind the proof of Theorem 2 can be summarized as follows: we start with the identity
$$ u(x) = \mathbb{E} \left[u(\omega^*_{x}(t \wedge \tau)) e^{- (t \wedge \tau)} \right].$$
In the proof of Theorem 2, we will abandon all control over the function $u$: our hope is that most Brownian particles will require a
long time until they hit $\partial E$. Formally, we will let $t \rightarrow \infty$ and argue that
$$ u(x) = \mathbb{E} \left[u(\omega^*_{x}(t \wedge \tau)) e^{- (t \wedge \tau)} \right] \leq \|u\|_{L^{\infty}(\partial E)} \cdot \mathbb{E} \left[e^{-  \tau} \right].$$
We note that if $u$ is roughly comparable to $\|u\|_{L^{\infty}(\partial E)}$ on $\partial E$ on a set of substantial measure, then this estimate is not all that lossy. It has replaced our problem of estimating $u$ with a new problem: estimating $\mathbb{E} \left[e^{-  \tau} \right]$, where $\tau$ is the first hitting time of a (rescaled) Brownian particle started in $x$ on $\partial E$.
We first note an important aspect of the expression $\mathbb{E} \left[e^{-  \tau} \right]$: due to the exponential weight, the expression does not at all care about paths that eventually arrive at $\partial E$ after a long time has passed. It only cares about the small number of paths that arrive at $\partial E$ relatively quickly. This, in essence, is one of the reasons why Agmon's estimate is so successful in a large number of cases: the reduction to a one-dimensional path may be lossy but is not overall lossy with respect to the relatively tiny fraction of particles that traverse along that particular path; the particles who deviate away from the path and take another route to $\partial E$ require, in all likelihood, more time and therefore do not contribute substantially to the expectation.

\subsection{A transplantation procedure.} We proceed as follows: instead of working directly with $\mathbb{E} \left[e^{-  \tau} \right]$, we take a particular particle and track the behavior of the Agmon metric under its flow. We recall that $\tau$ is the first time the particle hits $\partial E$, which, equivalently, is first time that the Agmon metric assumes the value 0. We start with a local Taylor expansion.
 Let $f:\mathbb{R}^n \rightarrow \mathbb{R}$ be arbitrary and let $\omega_x^*(t)$ be the rescaled Brownian motion. Since we make no particular assumptions about $f$, we may assume w.l.o.g. that $x=0$. Then, for $t$ small,
$$ f( \omega^*_0(t)) = f(0) + \left\langle \nabla f, \omega^*_0(t) \right\rangle + \frac{1}{2} \left\langle \omega^*_0(t), (D^2 f(0)) \omega^*_0(t)\right\rangle + \mbox{l.o.t.}$$
We note that the linear term behaves like a standard one-dimensional (rescaled) Brownian motion scaled by $\| \nabla f\|$. Since $\nabla f$ is a vector of fixed size, we can write
$$  \left\langle \nabla f, \omega^*_0(t) \right\rangle = \| \nabla f \|  \left\langle \frac{\nabla f}{\| \nabla f\|}, \omega^*_0(t) \right\rangle$$
and use the fact that the inner product of an $n-$dimensional Brownian motion and a fixed vector of size 1 behaves like a one-dimensional Brownian motion. Here, we end up a one-dimensional Brownian motion scaled by $\| \nabla f\|$.
 The quadratic term has an expectation that can be explicitly computed. We start by computing it for the classical Brownian motion $\omega_0(t)$.
$\omega_0(t)$ is distributed according to a Gaussian centered at $0$: we may thus write 
$$ \omega_0(t) = \left(\gamma_1, \dots, \gamma_n\right) \qquad \mbox{where} \quad \gamma_i \sim \sqrt{t} \cdot \mathcal{N}(0,1)$$
are $n$ independent and identically distributed Gaussian random variables. Then
\begin{align*}
 \mathbb{E} ~ \frac{1}{2} \left\langle \omega_0(t), (D^2 f(0)) \omega_0(t)\right\rangle &= \mathbb{E}~ \frac{1}{2} \sum_{i,j=1}^{n} (D^2 f(0))_{ij} \gamma_i \gamma_j \\
 &= \frac12 \sum_{i=1}^{n} (D^2 f(0))_{ii} \cdot \mathbb{E} \gamma_i^2 \\
 &= \frac{t}{2} \mbox{tr} (D^2 u(0)) = \frac{t}{2} \Delta u(0).
\end{align*}
Since $\omega_x^*(t)$ is, locally, merely a rescaling of $\omega_x(t)$ we have, for $t$ small,
$$ \mathbb{E}  \frac{1}{2} \left\langle \omega^*_x(t), (D^2 f(x_0)) \omega^*_x(t)\right\rangle  = \frac{t}{2} \frac{1}{V(x) - \lambda} (\Delta f(x)) + o(t).$$

In particular, if $\Delta f(x) \geq 0$, then we observe that the primary driving factor is diffusion while the lower-order factor (which acts as transport) is moving mass further outside in the sense that we have, for $t$ small,
\begin{align*}
 \mathbb{E} f(\omega^*_x(t)) &= f(x) +  \left\langle \nabla f(x), \omega^*_0(t) \right\rangle +  \frac{t}{2} \frac{1}{V(x) - \lambda} \Delta f(x) + o(t) \\
 &\geq f(x) + \left\langle \nabla f(x), \omega^*_0(t) \right\rangle.
\end{align*}
We will now use this inequality of a particular choice of $f$
$$ f(x) = \rho_{\lambda}(E,x).$$
As we just discussed, if $\Delta f\geq 0$, then we may interpret the process $f(\omega_x^*(t))$ as being primarily a one-dimensional random walk, this is the term $ \left\langle \nabla f(x), \omega^*_0(t) \right\rangle$ which we expect to be at scale $\sim \sqrt{t}$, and a higher-order term at scale $\sim t$ whose only effect is to \textit{increase} $f(\omega_x^*(t))$ (in expectation). By ignoring this higher-order term, we end up with a different stochastic process whose expected hitting time strictly dominates our original underlying process: this new process will hit 0 faster. In physical terms, we are dealing with diffusion coupled with a transport term where the transport term has the effect of moving mass further away from the origin. Hitting times are thus stochastically dominated by the pure diffusion term. However, a pure stochastic diffusion is easier to analyze than a diffusion with drift. We note that our diffusion has a coefficient depending on location. It remains to understood the relevant scaling: recalling the definition of the Agmon metric
$$ \rho_{\lambda}(x,y) = \inf_{\gamma} \int_0^1 \max\left( \sqrt{2} \sqrt{V(\gamma(t))- \lambda}, 0\right) |\dot \gamma(t)| dt,$$
we see that, at least in places where the geodesic is locally unique,
$$ \| \nabla \rho_{\lambda}(E,x)\| = \sqrt{2} \sqrt{V(x)-\lambda}.$$
At the same time, the rescaled Brownian motion moves locally rescaled by the factor $\sqrt{V(x) - \lambda}$ and we see that the two terms undergo an exact cancellation: the resulting effective speed of $\rho_{\lambda}(E,\omega_x^*(t))$ as a real-valued diffusion is \textit{constant} diffusion with speed $\sqrt{2}$. This very nice fact could be used to work backwards and would allow for a natural derivation of the Agmon metric from scratch. Our problem has now simplified: given a point $x > 0$ on the real line and given a pure diffusion process moving with \textit{constant} speed $\sqrt{2}$, what can be said about $\mathbb{E}\left( e^{-\tau_2}\right),$ where $\tau_2$ is the first hitting time of 0? This relates to our original question via stochastic domination as
$$ \mathbb{E} \left[e^{-  \tau} \right] \leq \mathbb{E} \left[e^{-  \tau_2} \right].$$

\subsection{Final Computations.} This problem is classical: we use \cite[Eq. 2.1]{ham}. This identity, valid for any $\lambda > 0$, and classical (unscaled) diffusion started in $x$ is
$$ \mathbb{E} \left(e^{-\lambda\cdot \tau_x}\right) = \frac{\sqrt{2} \sqrt{x \sqrt{2 \lambda}}}{\sqrt{\pi}} \cdot K_{-1/2}(x \sqrt{2\lambda}),$$
where $K_{\nu}(x)$ is the modified Bessel function of the second kind, a solution of
$$ z^2 \frac{d^2 w}{dz^2} + z \frac{dw}{dz} - (z^2+\nu^2)w = 0.$$
Since we have (spatial) diffusion speed $\sqrt{2}$, this corresponds to a time rescaling by a factor of $2$ and requires us to use $\lambda=1/2$. Then
$$  \mathbb{E} \left[e^{-  \tau_2} \right] = \frac{\sqrt{2} \sqrt{x}}{\sqrt{\pi}} \cdot K_{-1/2}(x).$$
$K_{-1/2}(x)$ is a special kind of Bessel function and has a simpler expression
$$ K_{-1/2}(x) = \sqrt{\frac{\pi}{2}} \frac{e^{-x}}{\sqrt{x}}.$$
It remains to specify $x$: our initial starting point for the diffusion is simply given by Agmon's distance $\rho_{\lambda}(E,x)$. Therefore,
$$  \mathbb{E} \left[e^{-  \tau_2} \right] \leq e^{-\rho_{\lambda}(E, x)}$$
and thus
\begin{align*}
 |u(x)|  &\leq \|u\|_{L^{\infty}} \cdot  \mathbb{E} \left[e^{-  \tau} \right] \\
 &\leq  \|u\|_{L^{\infty}} \cdot  \mathbb{E} \left[e^{-  \tau_2} \right]  \leq \|u\|_{L^{\infty}} \cdot e^{-\rho_{\lambda}(E, x)}.
 \end{align*}
This concludes the proof of Theorem 3.

\subsection{Variations.} The argument naturally allows for variations. 
Running through the proof of Theorem 3, we see that our transplantation procedure ends up resulting in a pure (rescaled) diffusion on the real
line coupled with a transport term: for small values of $t$, we have
$$ \mathbb{E} \rho_{\lambda}(E,\omega^*_x(t)) = \rho_{\lambda}(E,x) +  \left\langle \nabla  \rho_{\lambda}(E,x), \omega^*_0(t) \right\rangle +  \frac{t}{2} \frac{1}{V(x) - \lambda} (\Delta  \rho_{\lambda}(E,x)) + o(t).$$
We know that $\left\langle \nabla  \rho_{\lambda}(E,x), \omega^*_0(t) \right\rangle$ behaves like classical Euclidean diffusion scaled up by a factor of $\sqrt{2}$ in space or, equivalently, a factor of $2$ in time. If the quadratic transport term assumes negative values, then we can still compare the arising process with other stochastic process, for example with the Bessel process whose 
$$\mbox{infinitesimal generator is} \quad \frac{1}{2} \frac{d^2}{dx^2} + \frac{2\nu + 1}{2x} \frac{d}{dx}.$$
We see that the Bessel process also has negative transport (and thus smaller hitting times) when $\nu < -1/2$. 
However, the relevant identities still exist, see e.g. \cite[Eq. 2.4]{ham}, and we can deduce (for the Bessel process) that
$$ \mathbb{E}\left[ e^{- \tau_x/2}\right] = \frac{2^{\nu+1}}{\Gamma(|\nu|) x^{\nu}}  \cdot K_{\nu}(x).$$
This part of the argument is rather flexible and one could use many other stochastic processes: each of these process
can naturally induce a stochastic domination by comparing the second order expansion
$$ \mathbb{E} \rho_{\lambda}(E,\omega^*_x(t)) = \rho_{\lambda}(E,x) +  \left\langle \nabla  \rho_{\lambda}(E,x), \omega^*_0(t) \right\rangle +  \frac{t}{2} \frac{1}{V(x) - \lambda} (\Delta  \rho_{\lambda}(E,x)) + o(t)$$
with the infinitesimal generator of the process and using domination of the transport term as a condition. If, for this process,
there is a bound on $ \mathbb{E}\left[ e^{-  \lambda \cdot \tau_x}\right]$, the bound then naturally transfers to our setting.

\section{Proof of Theorem 2} \label{proof:2}
\begin{proof} We use again the identity
$$ |u(x)| \leq \|u\|_{L^{\infty}(\partial E)} \cdot \mathbb{E} \left[ e^{-  \tau} \right].$$
Invoking harmonic measure, we know that the likelihood of a particle hitting $\partial E$ at all is given by 
$$ \mathbb{P}(\tau < \infty) = \omega_x^{(\mathbb{R}^n \setminus E)}(\partial E).$$
This allows for an immediate proof of the trivial upper bound
\begin{align*}
 |u(x)| &\leq \|u\|_{L^{\infty}(\partial E)} \cdot \mathbb{E} \left[ e^{-  \tau} \right] \\
 &\leq  \|u\|_{L^{\infty}(\partial E)} \mathbb{P}(\tau < \infty) =   \|u\|_{L^{\infty}(\partial E)} \cdot \omega_x^{(\mathbb{R}^n \setminus E)}(\partial E).
\end{align*}
It is now natural to ask whether we can improve on this trivial estimate. Clearly, the expected likelihood is maximized if the particles that end up hitting $\partial E$ are simultaneously the ones that do it the fastest -- and the trivial estimate follows from assuming that they do so instantaneously which is clearly not the case. 
This suggest that we should try to understand the distribution of arrival time -- as in the classical Agmon estimate, we ignore the transport term as a lower-order contribution. 
$ t \rightarrow \rho_{\lambda}(E, \omega_x^*(t))$ 
behaves like classical diffusion started at distance $\rho_{\lambda}(E, x)$  from the origin scaled up by a factor of $\sqrt{2}$ in space (or, equivalently, a factor of 2 in time). The first passage time of classical diffusion is then given by
$$ \psi(t) = \frac{\rho_{\lambda}(E, x)}{\sqrt{4 \pi  t^3}} \exp\left(-\frac{\rho_{\lambda}(E,x)^2}{4 t} \right).$$
In particular, we have
$$ \int_0^{\infty} \psi(t) dt = 1$$
as well as
$$ \int_0^{\infty} e^{-t} \psi(t) dt = \exp\left( - \rho_{\lambda}(E,x)\right)$$
which recovers the classical Agmon decay rate: indeed, this can be seen as an alternative argument leading to the Agmon estimate. Our remaining question is: if $T$ is chosen such that
$$ \int_0^{T} \psi(t) dt = \omega_x^{(\mathbb{R}^n \setminus E)}(\partial E),$$
what can be said about the expected size of the integral
$\int_0^{T} e^{-t} \psi(t) dt?$
We first observe that
$$ \omega_x^{(\mathbb{R}^n \setminus E)}(\partial E) = \int_0^{T} \psi(t) dt  = \erfc\left(\frac{ \rho_{\lambda}(E,x)}{2\sqrt{T}}\right),$$
where $\erfc$ is the complementary error function. If $x \sim 1$, then $\erfc(x) \sim 1$, the interesting case is when the harmonic
measure is rather small. In that case, we have the asymptotic expansion capturing the leading order asymptotic
$$ \erfc(x) \sim \frac{e^{-x^2}}{\sqrt{\pi} x}.$$
Thus, again up to leading order,
$$ \frac{\rho_{\lambda}(E,x)^2}{4T} \sim \log\left(\frac{1}{\omega_x^{(\mathbb{R}^n \setminus E)}(\partial E)} \right)$$
which gives us, to leading order, insight into the size of $T$. A useful identity is
\begin{align*}
  \int_0^{T} e^{-t} \psi(t) dt  &= \frac{ \exp\left( - \rho_{\lambda}(E,x)\right)}{2} \erfc\left(\frac{ \rho_{\lambda}(E,x)}{2 \sqrt{T}} - \sqrt{T} \right) \\
  &+ \frac{ \exp\left(  \rho_{\lambda}(E,x)\right)}{2}  \erfc\left(\frac{ \rho_{\lambda}(E,x)}{2 \sqrt{T}} + \sqrt{T} \right).
\end{align*}
A simple analysis of the function
$$ \frac{e^{-r}}{2} \erfc\left( \frac{r}{2 \sqrt{t}} - \sqrt{t} \right) + \frac{e^r}{2} \erfc \left( \frac{r}{2\sqrt{t}} + t\right)$$
shows that the function is essentially $e^{-r}$ if $r \lesssim 2t$ and decays dramatically faster than $e^{-r}$ for $r \gtrsim 2t$. The transition region is $r \sim t \pm \sqrt{t}$.  This can be made precise as follows: if $r \gtrsim 2t$, then both $\erfc-$terms have positive arguments and we may use the asymptotic
 $\erfc(x) \sim e^{-x^2}$, to arrive at
$$   \int_0^{T} e^{-t} \psi(t) dt  \sim \exp\left( -\frac{ \rho_{\lambda}(E,x)^2}{4T} - T\right).$$
Observe that 
$$ \frac{ \rho_{\lambda}(E,x)^2}{4T} + T \geq  \rho_{\lambda}(E,x) \qquad \mbox{with equality iff} \qquad T = \frac{ \rho_{\lambda}(E,x)}{2}.$$
Therefore, a necessary condition of the Agmon estimate being sharp is that 
$ 2T \gtrsim \rho_{\lambda}(E, x)$
from which we deduce
$$  \log\left(\frac{1}{\omega_x^{(\mathbb{R}^n \setminus E)}(\partial E)} \right) \sim  \frac{\rho_{\lambda}(E,x)^2}{4T} \gtrsim  \frac{ \rho_{\lambda}(E,x)}{2}$$
from which the desired claim follows.
\end{proof}

\textbf{Remark.} We note that the argument is accurate up to leading order but is ignoring possible lower order effects due to transport: this effect can either be
absorbed in rescaling diffusion a little (leading to $(1-\varepsilon)$ factors as in the typical Agmon estimate) or by controlling
$\Delta \rho_{\lambda}(E_{\delta}, x)$ as in the proofs of Theorem 3.


\begin{thebibliography}{10}

\bibitem{agmon1} S. Agmon, Lectures on exponential decay of solutions of second order elliptic equations. Bounds
on eigenfunctions of N-body Schr\"odinger operators. Mathematical Notes, Princeton Univ. Press,
Princeton, N.J., 1982.

\bibitem{aiz} M. Aizenman and B. Simon, Brownian motion and Harnack's inequality for Schr\"odinger
operators, Comm. Pure Appl. Math. 35 (1982), 209--271.

\bibitem{ake} J. R. Akeroyd. Champagne subregions of the disk whose bubbles carry harmonic measure.
Math. Ann., 323(2):267--279, 2002. 

\bibitem{bet} D. Betsakos,
Harmonic measure on simply connected domains of fixed inradius, Ark. Mat. 36 (1998), no. 2, 275--306.

\bibitem{carmona0} R. Carmona, 
Pointwise bounds for Schr\"odinger eigenstates.
Comm. Math. Phys. 62 (1978), no. 2, 97--106.

\bibitem{carmona} R. Carmona and B. Simon, Pointwise Bounds on Eigenfunctions
and Wave Packets in $N-$Body Quantum Systems, Commun. Math. Phys. 80, 59-98 (1981)

\bibitem{carroll} T. Carroll and J. Ortega-Cerda, Configurations of balls in Euclidean space that Brownian motion cannot avoid. Ann. Acad. Sci. Fenn. Math. 32 (2007), no. 1, p. 223--234.

\bibitem{deift} P. Deift, Review of \cite{agmon1}, Bull. Amer. Math. Soc. 12, p. 165--169 (1985) 

\bibitem{denisov} S. Densiov and S. Kupin, Ito Diffusions, Modified Capacity, and Harmonic Measure. Applications to Schr\"odinger Operators. Inter. Math. Res. Not. 14 (2012): 3127--3181.


\bibitem{filoche} M. Filoche and S. Mayboroda, Universal mechanism for Anderson and weak localization. Proc. Natl. Acad. Sci. USA 109 (2012), 14761--14766.

\bibitem{gardiner} S. J. Gardiner and M. Ghergu. Champagne subregions of the unit ball with unavoidable
bubbles. Ann. Acad. Sci. Fenn. Math., 35(1):321--329, 2010


\bibitem{ham} Y. Hamana and H. Matsumoto, The probability distributions of the first hitting times of Bessel processes, Trans. Amer. Math. Soc. 365 (2013), 5237-5257


\bibitem{hislop} P. Hislop, Exponential decay of two-body eigenfunctions: A review, Proceedings of the Symposium on Mathematical Physics and Quantum Field Theory (Berkeley, CA, 1999). 


\bibitem{landis} V. A. Kondratiev, E. M. Landis, Qualitative theory of second order linear partial
differential equations, Partial differential equations -- 3, Itogi Nauki i Tekhniki,
Ser. Sovrem. Probl. Mat. Fund. Napr., 32, VINITI, Moscow, 1988, 99--215.



\bibitem{lithner} L. Lithner, 
A theorem of the Phragmen-Lindelof type for second-order elliptic operators.
Ark. Mat. 5 (1964), 281--285 (1964).

\bibitem{meshkov} V. Z. Meshkov, On the possible rate of decay at infinity of solutions of second
order partial differential equations, Math. USSR Sbornik 72, 1992, 343--360

\bibitem{ortega}  J. Ortega-Cerda and K. Seip. Harmonic measure and uniform densities. Indiana Univ. Math.
J., 53(3), p. 905--923, 2004.


\bibitem{stein} S. Steinerberger, Localization of quantum states and landscape functions. Proc. Amer. Math. Soc. 145 (2017), no. 7, 2895--2907. 

\bibitem{stein1} S. Steinerberger, Regularized Potentials of Schr\"odinger Operators and a Local Landscape Function, Comm. PDE 46, 1262-1279 (2021).

\bibitem{thal} A. Thalmaier and J. Thompson,
Exponential integrability and exit times of diffusions on sub-Riemannian and metric measure spaces. Bernoulli 26 (2020), no. 3, 2202--2225.

\end{thebibliography}
\end{document}